%% file: main.tex
\documentclass[oneside,abstracton]{scrartcl}
\usepackage[utf8]{inputenc}
\usepackage[T1]{fontenc}
\usepackage{lmodern}
\usepackage[english]{babel}

\usepackage[intlimits]{amsmath}
\usepackage{amstext}
\usepackage{amssymb}
\usepackage[reftex]{theoremref}
\usepackage{enumerate}
\usepackage{cite}
\usepackage{amsthm}

\input{macros}

\swapnumbers

\theoremstyle{plain}
	\newtheorem{theorem}{Theorem}[section]
	\newtheorem{lemma}[theorem]{Lemma}
	\newtheorem{corollary}[theorem]{Corollary}
	
	\newtheorem*{theorem*}{Theorem}

\theoremstyle{remark}
	\newtheorem{remark}[theorem]{Remark}

\title{$\sln$-Covariant $\Lp$-Minkowski Valuations}
\author{Lukas Parapatits}
\date{}

\begin{document}
	\maketitle

	\begin{abstract}
		All continuous $\sln$-covariant $\Lp$-Minkowski valuations defined on convex bodies are completely classified.
		The $\Lp$-moment body operators turn out to be the nontrivial prototypes of such maps. \\[0.5cm]
	\end{abstract}

	\section{Introduction}
		\input{intro}

	\section{Background Material}
		\input{background}

	\section{$\sln$-Covariance}
		\input{prelimSLn}

	\section{Main Results on $\cK^n_o$}
		\input{proofPno}

	\section{Main Results on $\cK^n$}
		\input{proofPn}

	\section*{Acknowledgements}
		\input{acknowledgements}

	\bibliography{bibliography}
	\bibliographystyle{plain}

	\vspace{1cm}
	\begin{flushleft}
		Lukas Parapatits \\
		Vienna University of Technology \\
		Institute of Discrete Mathematics and Geometry \\
		Wiedner Hauptstraße 8-10/104 \\
		1040 Vienna, Austria \\
		lukas.parapatits@tuwien.ac.at \\
	\end{flushleft}
\end{document}

%% file: macros.tex
\newcommand{\cK}	{\mathcal K}

\newcommand{\cP}	{\mathcal P}
\newcommand{\cQ}	{\mathcal Q}
\newcommand{\C}[1][n]	{C_1(\R^{#1})}
\newcommand{\Cp}[1][n]	{C_p(\R^{#1})}
\newcommand{\gln}[1][n]	{\operatorname{GL}(#1)}
\newcommand{\h}[2]	{h\!\left( #1 , #2 \right)}
\newcommand{\htiny}[2]	{h\left( #1 , #2 \right)}
\newcommand{\Lp}	{L_p}

\newcommand{\N}		{\mathbb N}
\newcommand{\oEpp}	{\operatorname{E}_p^+}
\newcommand{\oEpm}	{\operatorname{E}_p^-}
\newcommand{\oFpp}	{\operatorname{F}_p^+}
\newcommand{\oFpm}	{\operatorname{F}_p^-}
\newcommand{\oIpp}	{\operatorname{I}_p^+}
\newcommand{\oIpm}	{\operatorname{I}_p^-}
\newcommand{\oJpp}	{\operatorname{J}_p^+}
\newcommand{\oJpm}	{\operatorname{J}_p^-}
\newcommand{\oM}	{\operatorname{M}}
\newcommand{\oMp}	{\operatorname{M}_p}
\newcommand{\oMpp}	{\operatorname{M}_p^+}
\newcommand{\oMpps}	{\hat{\operatorname{M}}_p^+}
\newcommand{\oMpm}	{\operatorname{M}_p^-}
\newcommand{\oMpms}	{\hat{\operatorname{M}}_p^-}

\newcommand{\oPhi}	{\operatorname{\Phi}}
\newcommand{\oPsi}	{\operatorname{\Psi}}
\newcommand{\padd}	{+_p}

\newcommand{\R}		{\mathbb R}
\newcommand{\s}[2]	{\left\langle #1 , #2 \right\rangle}
\newcommand{\sln}[1][n]	{\operatorname{SL}(#1)}

\DeclareMathOperator{\vol}{vol}

\DeclareMathOperator{\aff}{aff}

\DeclareMathOperator{\conv}{conv}
\DeclareMathOperator{\lin}{lin}

\DeclareMathOperator{\sgn}{sgn}

%% file: intro.tex
A valuation is a finitely additive function on convex bodies,
i.e.\ nonempty compact convex subsets of $\R^n$.
More general, let $\cQ^n$ be a subset of $\cK^n$, the set of convex bodies in $\R^n$,
and let $A$ be an abelian monoid.
A map $\oPhi \colon \cQ^n \to A$ is called a valuation, if
\begin{equation}\label{eq: def val}
	\oPhi( K \cup L ) + \oPhi( K \cap L ) = \oPhi( K ) + \oPhi( L )
\end{equation}
for all $K,L \in \cQ^n$ with $K \cup L, K \cap L \in \cQ^n$.
At the beginning of the last century Dehn's solution of Hilbert's Third Problem sparked a lot of interest in valuations.
In the 1950's Hadwiger started a systematic study of valuations,
which resulted in his famous classification of all continuous rigid motion invariant real valued valuations
(see e.g. \cite{Schneider1993}).
Numerous results in the spirit of Hadwiger's theorem were established over the last years
(see e.g. \cite{Alesker1999, Alesker2001, Bernig2009, BernigFu2011,
Klain2000, Ludwig2002_2, LudwigReitzner1999, LudwigReitzner2010, McMullen1993}).

Apart from real valued valuations convex body valued valuations are the focus of increased attention
(see e.g. \cite{Haberl2009, Haberl2011, Haberl_2, HaberlParapatits_1, Kiderlen2006, Ludwig2002_1, Ludwig2003,
Ludwig2005, Ludwig2006, Lutwak1990, ParapatitsSchuster_1, SchneiderSchuster2006, Schuster2010, SchusterWannerer2012, Wannerer_1}).
The simplest way to define an addition on $\cK^n$ such that it becomes an abelian monoid is Minkowski addition,
i.e.\ $K+L := \{x+y : x \in K, y \in L\}$ for $K,L \in \cK^n$.
The corresponding Minkowski valuations and their generalizations attracted a lot of attention in recent years,
because they have applications in many areas such as
convexity, stochastic geometry, functional analysis and geometric tomography
(see e.g. \cite{CampiGronchi2002, CianchiLutwakYangZhang2009, FleuryGuedonPaouris2007, HaberlSchuster2009_1, HaberlSchuster2009_2, Gardner2006,
GardnerKiderlenMilanfar2006, Kiderlen2008, Lutwak1990, LutwakYangZhang2000_1, LutwakYangZhang2002_2, LutwakYangZhang2004,
ParapatitsSchuster_1, SchneiderWeil2008, SchusterWannerer2012, Steineder2008, YaskinYaskina2006, Wannerer_1, WernerYe2008, Zhang1999}).

This article focuses on valuations $\oPhi$ which are $\sln$-covariant,
i.e.\ $\oPhi(\phi K) = \phi \oPhi K$ for all $K \in \cK^n$ and $\phi \in \sln$.
Clearly the identity $K \mapsto K, \ K \in \cK^n$, and the reflection at the origin $K \mapsto -K, \ K \in \cK^n$,
are $\sln$-covariant Minkowski valuations.
Another important example is the moment body operator $\oM \colon \cK^n \to \cK^n_o$ defined by
	$$ \h{\oM K}{u} = \int_K |\s{x}{u}| \ dx, \quad u \in \R^n, $$
for all $K \in \cK^n$, where $\cK^n_o$ denotes the set of convex bodies containing the origin
and where $\h{K}{u} := \max_{x \in K} \s{x}{u}, \ u \in \R^n,$ is the support function of the convex body $K$.
For a volume normalized convex body this integral can be interpreted as the expectation of a certain mass distribution
depending on $K$ and $u$.
Moment bodies (under a different normalization) are also called centroid bodies
and were formally defined by Petty \cite{Petty1961},
but they actually date back to Dupin and Blaschke.
The name ``centroid body'' comes from the fact that for a symmetric convex body $K$
the boundary of $\oM K$ consists of points
which are, up to normalization, the centroids of $K$ intersected with halfspaces.
Ludwig \cite[Corollary 1.1]{Ludwig2005} showed that the only examples
of continuous homogeneous $\sln$-covariant Minkowski valuations on $\cK^n_o$
are either trivial, i.e.\ a combination of the identity and the reflection at the origin,
or a combination of the moment body operator and the moment vector operator.
Assuming $\sln$-covariance together with homogeneity is basically assuming $\gln$-covariance,
where one has to take the determinant of $\phi \in \gln$ into account.
Once the valuations that were $\gln$-covariant were understood,
the next step was in trying to classify those valuations
that were \textit{only} $\sln$-covariant.
This turned out to be difficult, because the involved functional equations were a lot more complicated.
A real breakthrough here was first achieved by Haberl in \cite{Haberl_2},
where he removed the assumption on homogeneity in Ludwig's classification.

A generalization of Minkowski addition is $L_p$-Minkowski addition
(also known as Minkowski-Firey $L_p$-addition) for $p > 1$,
which is defined by
\begin{equation}\label{eq: def padd}
	\h{K \padd L}{.}^p = \h{K}{.}^p + \h{L}{.}^p
\end{equation}
for all $K,L \in \cK^n_o$, the set of convex bodies containing the origin.
The corresponding $L_p$-Minkowski valuations also received a lot of attention in recent years
(see e.g. \cite{CampiGronchi2002, FleuryGuedonPaouris2007, LutwakYangZhang2000_2, GardnerGiannopoulos1999,
Haberl2008, HaberlLudwig2006, HaberlSchuster2009_2,
Ludwig2005, LutwakYangZhang2000_1, LutwakYangZhang2004, Parapatits_1,
WernerYe2008, YaskinYaskina2006}).
In particular, they play an important role in new affine Sobolev inequalities
(see e.g. \cite{CianchiLutwakYangZhang2009, HaberlSchuster2009_1, LutwakYangZhang2002_2, Zhang1999}).

The identity and the reflection at the origin, both restricted to $\cK^n_o$,
are also $\sln$-covariant $\Lp$-Minkowski valuations.
Furthermore, the (symmetric) $\Lp$-moment body operator $\oMp \colon \cK^n \to \cK^n_o$ defined by
	$$ \h{\oMp K}{u}^p = \int_K |\s{x}{u}|^p \ dx, \quad u \in \R^n $$
for all $K \in \cK^n$,
which encodes the $p$-th moment of the aforementioned mass distribution,
is an $\sln$-covariant $\Lp$-Minkowski valuation.
For the definitions of the (asymmetric) $\Lp$-moment body operators $\oMpp$ and $\oMpm$ see \eqref{eq: def oMpp}.
Ludwig \cite[Corollary after Theorem $1_p$]{Ludwig2005} also showed that the only nontrivial examples
of continuous homogeneous $\sln$-covariant $\Lp$-Minkowski valuations on $\cK^n_o$
are combinations of the two (asymmetric) $\Lp$-moment body operators.
Note that Ludwig uses the notation $M_p^{\tau} = (1+\tau) \oMpp \padd (1-\tau) \oMpm$ for $\tau \in [-1,1]$.
Our first main theorem improves this classification by removing the homogeneity assumption.

\begin{theorem*}
	Let $n \geq 3$.
	An operator $\oPhi \colon \cK^n_o \to \cK^n_o$ is
	a continuous $\sln$-covariant $\Lp$-Minkowski valuation,
	if and only if there exist constants $c_1, c_2, c_3, c_4 \geq 0$ such that
	\begin{equation*}
		\oPhi K = c_1 \oMpp K \padd c_2 \oMpm K \padd c_3 K \padd c_4 (-K)
	\end{equation*}
	for all $K \in \cK^n_o$.
\end{theorem*}

Wannerer \cite{Wannerer_1} extended Ludwig's characterization from valuations on $\cK^n_o$ to valuations on $\cK^n$.
In this case four additional operators arise.
Our second main theorem characterizes continuous $\sln$-covariant $\Lp$-Minkowski valuations defined on $\cK^n$.
Again, no assumption on homogeneity is needed.
In our case only two additional operators arise.
For the definitions of $\oMpps$ and $\oMpms$ see \eqref{eq: def oMpps}.
Furthermore, $K_o$ is defined as the convex hull of a convex body $K$ and the origin.

\begin{theorem*}
	Let $n \geq 3$.
	An operator $\oPhi \colon \cK^n \to \cK^n_o$ is
	a continuous $\sln$-covariant $\Lp$-Minkowski valuation,
	if and only if there exist constants $c_1, c_2, c_3, c_4, c_5, c_6 \geq 0$ such that
	\begin{equation*}
		\oPhi K = c_1 \oMpp K \padd c_2 \oMpm K \padd c_3 \oMpps K \padd c_4 \oMpms K \padd c_5 K_o \padd c_6 (-K_o)
	\end{equation*}
	for all $K \in \cK^n$.
\end{theorem*}

We remark that a complete classification of all continuous $\sln$-contravariant $\Lp$-Minkowski valuations
was also recently obtained by the author in \cite{Parapatits_1}.

%% file: background.tex
As a general reference for the material in this section see \cite{Gruber2007, KlainRota1997, Schneider1993}.
For the dimension of the Euclidean space $\R^n$, we will always assume that $n \geq 1$.
The standard basis vectors will be denoted by $e_1,\ldots,e_n$ and the origin by $o$.
We will often write $x = (x_1, \ldots, x_n)^t$ for $x \in \R^n$.
For $x,y \in \R^n$ the scalar product, the induced norm and the orthogonal complement will be denoted by
$\s{x}{y}$, $\|x\|$ and $x^\perp$, respectively.
The linear, affine and convex hull, i.e.\ all linear, affine and convex combinations of a given set,
are denoted by $\lin$, $\aff$ and $\conv$, respectively.
Furthermore we will assume $p > 1$ throughout this article.

Associated with a convex body $K$ is its dimension $\dim K$,
which is equal to the dimension of $\aff K$.
If $K$ is an $m$-dimensional convex body we denote its $m$-dimensional volume by $\vol_m K$.
Most of the time we will not work with valuations defined on $\cK^n$ or $\cK^n_o$,
but with valuations defined on convex polytopes,
i.e.\ convex hulls of finite subsets of $\R^n$.
The set of all convex polytopes is denoted by $\cP^n$
and the subset of all convex polytopes containing the origin by $\cP^n_o$.

Apart from Minkowski addition and $\Lp$-Minkowski addition, there is also a scalar multiplication for convex bodies.
It is defined by $sK = \{sx : x \in K\}$ for all $K \in \cK^n$ and $s \geq 0$.
The sets $\cK^n$, $\cK^n_o$, $\cP^n$ and $\cP^n_o$ are all closed under Minkowski addition and scalar multiplication.
Furthermore, the sets $\cK^n_o$ and $\cP^n_o$ are also closed under $\Lp$-Minkowski addition.
When we talk about continuity, we mean continuity with respect to the Hausdorff metric.
The Hausdorff distance between two convex bodies $K,L$ is defined by
\begin{equation*}
	\delta(K,L) = \min\{ \epsilon > 0 : K + \epsilon B^n \subseteq L, \ L + \epsilon B^n \subseteq K \},
\end{equation*}
where $B^n$ is the Euclidean unit ball in $\R^n$.
The sets $\cK^n$ and $\cK^n_o$ are closed in the Hausdorff topology.

We already mentioned the support function of a convex body in the introduction.
A convex body is uniquely defined by its support function.
On the other hand, a function $h \colon \R^n \to \R$ is the support function of a convex body,
if and only if it is sublinear, i.e.\
\begin{equation*}
	h(u+v) \leq h(u) + h(v) \quad \text{and} \quad h(su) = sh(u)
\end{equation*}
for all $u,v \in \R^n$ and $s > 0$.
The first property is called subadditivity and the second one is $1$-homogeneity.
To see that the definition of $\Lp$-Minkowski addition \eqref{eq: def padd} makes sense one just has to verify that
$\h{K \padd L}{.}$ is a nonnegative sublinear function for all $K,L \in \cK^n_o$.
Note that every sublinear function is convex and therefore continuous.

The map $K \mapsto \h{K}{.}$ is an injective homomorphism from $\cK^n$
to the space of $1$-homogeneous continuous functions on $\R^n$, denoted $\C$, i.e.\
\begin{equation*}
	\h{K+L}{.} = \h{K}{.} + \h{L}{.} \quad \text{and} \quad \h{sK}{.} = s\h{K}{.}
\end{equation*}
for all $K,L \in \cK^n$ and $s \geq 0$.
Analogously the map $K \mapsto \h{K}{.}^p$ is an injective homomorphism from $\cK^n_o$
to the space of $p$-homogeneous continuous functions on $\R^n$, denoted $\Cp$, i.e.\
\begin{equation*}
	\h{K \padd L}{.}^p = \h{K}{.}^p + \h{L}{.}^p \quad \text{and} \quad \h{sK}{.}^p = s^p \h{K}{.}
\end{equation*}
for all $K,L \in \cK^n_o$ and $s \geq 0$.
We can calculate the Hausdorff distance of two convex bodies $K,L$ by $\delta(K,L) = \| \h{K}{.} - \h{L}{.} \|_\infty$,
where $\| . \|_\infty$ denotes the maximum norm on the Euclidean unit sphere in $\R^n$, denoted $S^{n-1}$.
Since support functions are homogeneous, they are determined by their values on $S^{n-1}$.
Therefore, it makes sense to equip $\C$ and $\Cp$ with this norm.
This also makes the two homomorphisms from above continuous.

We have already defined the notion of valuation in the introduction \eqref{eq: def val}.
Note that if $\cQ^n$ equals $\cK^n$,$\cK^n_o$, $\cP$ or $\cP^n_o$ the definition simplifies a little,
because in this case $K \cup L \in \cQ^n$ implies $K \cap L \in \cQ^n$ for all $K,L \in \cQ^n$.
It is easy to see that a map $\oPhi \colon \cK^n \to \cK^n$ is a Minkowski valuation,
if and only if the map $K \mapsto \h{\oPhi K}{.}$ from $\cK^n$ to $\C$ is a valuation.
Similarly a map $\oPhi \colon \cK^n \to \cK^n_o$ is an $\Lp$-Minkowski valuation,
if and only if the map $K \mapsto \h{\oPhi K}{.}^p$ from $\cK^n$ to $\Cp$ is a valuation.
The same also holds for valuations defined on $\cK^n_o$, $\cP$ or $\cP^n_o$.

We now recall some general results on valuations.
The first theorem is due to Volland \cite{Volland1957} (see also \cite{KlainRota1997}).
Short proofs of the other results can be found in \cite{Parapatits_1}.

It is convenient to assume that $A$ contains an identity element denoted $0$
and to define that $\oPhi(\emptyset) = 0$, even though $\emptyset \not \in \cK^n$.
We will do this throughout this article.

\begin{theorem}\th\label{incl excl Pn}
	Let $A$ be an abelian group and let $\oPhi \colon \cP^n \to A$ be a valuation.
	Then $\oPhi$ satisfies the inclusion exclusion principle, i.e.\
	\begin{equation*}
		\oPhi(P_1 \cup \ldots \cup P_m) =
		\sum_{\emptyset \neq S \subseteq \{1,\ldots,m\}} (-1)^{|S|-1} \oPhi \left( \bigcap_{i \in S} P_i \right)
	\end{equation*}
	for all $m \in \N$ and $P_1, \ldots, P_m \in \cP^n$ with $P_1 \cup \ldots \cup P_m \in \cP^n$.
\end{theorem}

The convex hull of $k+1$ affinely independent points is called a $k$-dimensional simplex.
Special simplices are the $n$-dimensional standard simplex $T^n := \conv\{ o, e_1, \ldots, e_n \}$
and $\widetilde T^{n-1} := \conv\{ e_1, \ldots, e_n \}$, which is an $(n-1)$-dimensional simplex.

\begin{lemma}\th\label{valuation determined simplices Pn}
	Let $A$ be an abelian group and let $\oPhi \colon \cP^n \to A$ be a valuation.
	Then $\oPhi$ is determined by its values on $n$-dimensional simplices.
\end{lemma}

\begin{lemma}\th\label{incl excl Pno}
	Let $A$ be an abelian group and let $\oPhi \colon \cP^n_o \to A$ be a valuation.
	Then $\oPhi$ satisfies the inclusion exclusion principle, i.e.\
	\begin{equation*}
		\oPhi(P_1 \cup \ldots \cup P_m) =
		\sum_{\emptyset \neq S \subseteq \{1,\ldots,m\}} (-1)^{|S|-1} \oPhi \left( \bigcap_{i \in S} P_i \right)
	\end{equation*}
	for all $m \in \N$ and $P_1, \ldots, P_m \in \cP^n_o$ with $P_1 \cup \ldots \cup P_m \in \cP^n_o$.
\end{lemma}

\begin{lemma}\th\label{valuation determined simplices}
	Let $A$ be an abelian group and let $\oPhi \colon \cP^n_o \to A$ be a valuation.
	Then $\oPhi$ is determined by its values on $n$-dimensional simplices with one vertex at the origin
	and its value on $\{o\}$.
\end{lemma}

A valuation is called simple, if $\oPhi K = 0$ for all $K \in \cQ^n$ with $\dim K < n$.

\begin{lemma}\th\label{simple determined Pno}
	Let $A$ be an abelian group and let $\oPhi \colon \cP^n \to A$ be a simple valuation.
	Then $\oPhi$ is determined by its values on $\cP^n_o$.
\end{lemma}

Finally we recall Cauchy's functional equation
\begin{equation}\label{eq: cauchy}
	f(a+b) = f(a) + f(b) \quad \forall a,b \in \R.
\end{equation}
Of course every linear function satisfies \eqref{eq: cauchy}.
It is a well known fact that a nonlinear function $f \colon \R \to \R$ can only satisfy \eqref{eq: cauchy},
if it has a dense graph in $\R^2$ or equivalently if every open subset of $\R$ has a dense image under $f$.
A function $f \colon (0,+\infty) \to \R$ satisfying \eqref{eq: cauchy} for all $a,b \in (0,+\infty)$
can be extended to an odd function on $\R$, which then satisfies Cauchy's functional equation for all $a,b \in \R$.
Therefore such a function $f$ is either linear or it has the property
that every open subset of $(0,+\infty)$ has a dense image under $f$.

%% file: prelimSLn.tex
Let $\cQ^n$ be $\cK^n$, $\cP^n$, $\cK^n_o$ or $\cP^n_o$.
A map $\oPhi \colon \cQ^n \to \cK^n_o$ is called $\sln$-covariant, if it satisfies
\begin{equation*}
	\oPhi( \phi K ) = \phi \oPhi K
\end{equation*}
for all $K \in \cQ^n$ and $\phi \in \sln$.
A map $\oPhi \colon \cQ^n \to \Cp$ is called $\sln$-covariant, if it satisfies
\begin{equation*}
	\oPhi( \phi K ) = \oPhi(K) \circ \phi^t
\end{equation*}
for all $K \in \cQ^n$ and $\phi \in \sln$.
Since
\begin{equation*}
	\h{\phi \oPhi K}{u} = \h{\oPhi K}{\phi^t u}
\end{equation*}
holds for all $K \in \cQ^n$, $u \in \R^n$ and $\phi \in \sln$,
we see that a map $\oPhi \colon \cQ^n \to \cK^n_o$ is $\sln$-covariant,
if and only if $K \mapsto \h{\oPhi K}{.}^p$ from $\cQ^n$ to $\Cp$ is $\sln$-covariant.
We also make analogous definitions and remarks for the general linear group, $\gln$.

We define an $\sln$-covariant $\Lp$-Minkowski valuation $\oIpp \colon \cK^n \to \cK^n_o$ by
\begin{equation*}
	\h{\oIpp K}{.}^p = \max_{x \in K} \s{x}{.}_+^p
\end{equation*}
for all $K \in \cK^n$, where $\s{x}{.}_+ := \max(0,\s{x}{.})$ denotes the positive part of $\s{x}{.}$.
Note that $\oIpp K = K_o$ for all $K \in \cK^n$.
Similarly we define an $\sln$-covariant $\Lp$-Minkowski valuation $\oIpm \colon \cK^n \to \cK^n_o$ by
\begin{equation*}
	\h{\oIpm K}{.}^p = \max_{x \in K} \s{x}{.}_-^p
\end{equation*}
for all $K \in \cK^n$, where $\s{x}{.}_- := \max(0,-\s{x}{.})$ denotes the negative part of $\s{x}{.}$.
Note that $\oIpm K = -K_o$ for all $K \in \cK^n$.
There are two other very similar $\sln$-covariant valuations,
namely $\oJpp \colon \cK^n \to \Cp$ and $\oJpm \colon \cK^n \to \Cp$.
These are defined by
\begin{equation*}
	\oJpp K = \min_{x \in K} \s{x}{.}_+^p
\end{equation*}
for all $K \in \cK^n$ and by
\begin{equation*}
	\oJpm K = \min_{x \in K} \s{x}{.}_-^p
\end{equation*}
for all $K \in \cK^n$, respectively.
Note that $\oJpp K$ and $\oJpm K$ are not necessarily support functions of convex bodies.
Also note that $\oJpp$ and $\oJpm$ vanish on $\cK^n_o$.
Finally we remark that $\oIpp$, $\oIpm$, $\oJpp$ and $\oJpm$ are also continuous and $\gln$-covariant.

Another family of $\sln$-covariant $\Lp$-Minkowski valuations are the $\Lp$-moment body operators.
We define $\oMpp \colon \cK^n \to \cK^n_o$ by
\begin{equation}\label{eq: def oMpp}
	\h{\oMpp K}{.}^p = \int_K \s{x}{.}_+^p dx
\end{equation}
for all $K \in \cK^n$.
Similarly we define $\oMpm \colon \cK^n \to \cK^n_o$.
Note that $\oMpp$ and $\oMpm$ are simple.
A variant of $\oMpp$ is $\oMpps \colon \cK^n \to \cK^n_o$ defined by
\begin{equation}\label{eq: def oMpps}
	\h{\oMpps K}{.}^p = \int_{K_o \setminus K} \s{x}{.}_+^p dx
\end{equation}
for all $K \in \cK^n$.
Similarly we define $\oMpms \colon \cK^n \to \cK^n_o$.
Note that $\oMpps$ and $\oMpms$ vanish on $\cK^n_o$.
Finally we remark that $\oMpp$, $\oMpm$, $\oMpps$ and $\oMpms$ are also continuous and $(n+p)/p$-homogeneous.

Recalling the embedding of $\cK^n_o$ into $\Cp$,
we have defined eight linearly independent $\sln$-covariant valuations with values in $\Cp$.
Notice that evaluating the minus-version of one of these operators at $K$ is the same as evaluating the plus-version at $-K$,
for example $\oIpm K = \oIpp (-K)$.

For later use we need to calculate some constants.
Let $i \in \{1, \ldots, n\}$.
We start with
\begin{equation}\label{eq: constant oIpp [o,e_i]}
	\h{\oIpp [o,e_i]}{x}^p = \max_{y \in [o,e_i]} \s{y}{x}_+^p = \s{e_i}{x}_+^p = (x_i)_+^p
\end{equation}
for all $x \in \R^n$.

Next we calculate
\begin{equation}\label{eq: constant oMpp T^n e_i}
\begin{split}
	\h{\oMpp T^n}{e_i}^p
	&= \int_{T^n} \s{x}{e_i}_+^p dx \\
	&= \int_0^1 x_i^p \vol_{n-1} \left( (1-x_i) T^{n-1} \right)  dx_i \\
	&= \vol_{n-1} \left( T^{n-1} \right) \int_0^1 x_i^p (1-x_i)^{n-1}  dx_i \\
	&= \frac 1 {(n-1)!} B(p+1,n) \\
	&= \frac {\Gamma(p+1) \Gamma(n)} {(n-1)! \ \Gamma(p+1+n)} \\
	&= \frac {\Gamma(p+1)} {\Gamma(p+1+n)} ,
\end{split}
\end{equation}
where $B$ and $\Gamma$ denote the Beta function and the Gamma function, respectively.
Obviously we have
\begin{equation}\label{eq: constant oMpp T^n -e_i}
	\h{\oMpp T^n}{-e_i}^p = 0 .
\end{equation}

Now we will calculate some constants which will be used in the classification of valuations on $\cP^n$.
Let $i \in \{1, \ldots, n\}$.
We start with
\begin{equation}\label{eq: constant oIpp [e_i, 2e_i]}
	\h{\oIpp [e_i, 2e_i]}{x}^p = \max_{y \in [e_i, 2e_i]} \s{y}{x}_+^p = \s{2e_i}{x}_+^p = 2^p (x_i)_+^p
\end{equation}
for all $x \in \R^n$.
Analogously we see that
\begin{equation}\label{eq: constant oJpp([e_i, 2e_i])}
	\oJpp([e_i, 2e_i])(x) = (x_i)_+^p .
\end{equation}

Next we calculate
\begin{equation}\label{eq: constant oMpps widetilde T^(n-1) e_i}
	  \h{\oMpps \widetilde T^{n-1}}{e_i}^p
	= \h{\oMpp T^n}{e_i}^p
	= \frac {\Gamma(p+1)} {\Gamma(p+1+n)}
\end{equation}
and
\begin{equation}\label{eq: constant oMpps widetilde T^(n-1) -e_i}
	\h{\oMpps \widetilde T^{n-1}}{-e_i}^p = 0 .
\end{equation}

Now let $n \geq 2$.
We set $K = \conv \{ e_1, e_2, e_1+e_2\}$ and calculate:
\begin{equation}\label{eq: constant oIpp oJpp conv (e_1, e_2, e_1+e_2)}
\begin{split}
	\h{\oIpp K}{e_1}^p = 1 ,\quad \h{\oIpp K}{e_2}^p = 1 ,\quad \h{\oIpp K}{e_1+e_2}^p = 2^p , \\
	\oJpp(K)(e_1) = 0 ,\quad \oJpp(K)(e_2) = 0 ,\quad \oJpp(K)(e_1+e_2) = 1 .
\end{split}
\end{equation}
The values of $\h{\oIpm K}{.}^p$ and $\oJpm K$ in these directions are all equal to $0$.
Finally we compute:
\begin{equation}\label{eq: constant oIpp oJpp widetilde T^1}
\begin{split}
	\h{\oIpp \widetilde T^1}{e_1}^p = 1 ,\quad \h{\oIpp \widetilde T^1}{e_1+e_2}^p = 1 ,\quad \h{\oIpp \widetilde T^1}{2e_1+e_2}^p &= 2^p , \\
	\oJpp(\widetilde T^1)(e_1) = 0 ,\quad \oJpp(\widetilde T^1)(e_1+e_2) = 1 ,\quad \oJpp(\widetilde T^1)(2e_1+e_2) = 1 .
\end{split}
\end{equation}
The values of $\h{\oIpm \widetilde T^1}{.}^p$ and $\oJpm \widetilde T^1$ in these directions are all equal to $0$.

In $\R^2$ there are other $\sln$-covariant valuations, which do not show up in $\R^n$ for $n \geq 3$.
We will only cover those additional operators in $\R^2$ that we need for the proof of the $n \geq 3$ case.
We define $\oEpp \colon \cK^2 \to \cK^2_o$ by
\begin{equation*}
	\h{\oEpp K}{.}^p = \frac 1 2 \sum_{\substack{u \in S^1 \\ \htiny{K}{u} = 0}} \max_{x \in F(K,u)} \s{x}{.}_+^p
\end{equation*}
for all $K \in \cK^n$, where $F(K,u) := \{x \in K : \s{x}{u} = \h{K}{u}\}$.
Note that this sum has at most two nonzero summands.
Clearly $\oEpp$ is $\gln$-covariant.
The valuation property can be shown by a case-by-case analysis.
Similarly we define $\oEpm \colon \cK^2 \to \cK^2_o$.
It is easy to see that $\oEpp$ and $\oEpm$ coincide with $\oIpp$ and $\oIpm$, respectively,
for $1$-dimensional convex bodies $K$ with $o \in \aff K$.

Finally $\oFpp \colon \cK^2 \to \Cp[2]$ is defined by
\begin{equation*}
	\oFpp K = \frac 1 2 \sum_{\substack{u \in S^1 \\ \htiny{K}{u} = 0}} \min_{x \in F(K,u)} \s{x}{.}_+^p
\end{equation*}
for all $K \in \cK^n$.
Similarly we define $\oFpm \colon \cK^2 \to \Cp[2]$.
Note that $\oFpp K$ and $\oFpm K$ are not necessarily support functions of convex bodies.
Also note that $\oFpp$ and $\oFpm$ vanish on $\cK^n_o$.
It is easy to see that $\oFpp$ and $\oFpm$ coincide with $\oJpp$ and $\oJpm$, respectively,
for $1$-dimensional convex bodies $K$ with $o \in \aff K$.
Notice again how one can get the minus-version of an operator by inserting $-K$ in the plus-version.

We will now collect some properties of these new valuations in $\R^2$ and calculate some constants.
Let $i \in \{1,2\}$.
We start with calculating
\begin{equation*}
	\h{\oEpp T^2}{e_i}^p = \frac 1 2 \left( \h{[o,e_1]}{e_i}^p + \h{[o,e_2]}{e_i}^p \right) = \frac 1 2.
\end{equation*}
Similarly we have
\begin{equation*}
	\h{\oEpp T^2}{-e_i}^p = 0 .
\end{equation*}
Therefore we get:
\begin{equation}\label{eq: constants oIpp - oEpp T^2 +-e_i}
\begin{split}
	\h{\oIpp T^2}{e_i}^p - \h{\oEpp T^2}{e_i}^p = 1 - \frac 1 2 = \frac 1 2 , \\
	\h{\oIpp T^2}{-e_i}^p - \h{\oEpp T^2}{-e_i}^p = 0 - 0 = 0.
\end{split}
\end{equation}

Next, we will look at the continuity of the above operators.
Clearly
\begin{equation*}
	\lim_{\varepsilon \to 0} [-\varepsilon e_1, e_1] + [-\varepsilon e_2, \varepsilon e_2] = [o, e_1] .
\end{equation*}
Since
\begin{equation*}
	\lim_{\varepsilon \to 0} \h{\oEpp \left( [-\varepsilon e_1, e_1] + [-\varepsilon e_2, \varepsilon e_2] \right)}{e_1}^p
	= \lim_{\varepsilon \to 0} \h{\{o\}}{e_1}^p
	= 0
\end{equation*}
and since
\begin{equation*}
	\h{\oEpp [o, e_1]}{e_1}^p = 1 ,
\end{equation*}
we see that
\begin{equation*}
	P \mapsto \h{\oEpp P}{.}^p
\end{equation*}
from $\cP^n_o$ to $\Cp$ is not continuous at $[o,e_1]$.
Similarly we prove the following lemma.

\begin{lemma}\th\label{cP^2_o not continuous}
	The only linear combination of
	\begin{equation*}
		P \mapsto \h{\oEpp P}{.}^p \quad \text{and} \quad P \mapsto \h{\oEpm P}{.}^p
	\end{equation*}
	which is continuous at $[o,e_1]$ is the trivial one.
\end{lemma}

Finally we need to calculate some constants related to valuations on $\cP^2$.
Let $i \in \{1,2\}$.
We start with
\begin{equation*}
	\h{\oEpp \widetilde T^1}{e_i}^p = \frac 1 2 \left( \s{e_1}{e_i}_+^p + \s{e_2}{e_i}_+^p \right) = \frac 1 2.
\end{equation*}
Similarly we have
\begin{equation*}
	\h{\oEpp \widetilde T^1}{-e_i}^p = 0 .
\end{equation*}
Therefore we get:
\begin{equation}\label{eq: constant oIpp - oEpp widetilde T^1 +-e_i}
\begin{split}
	\h{\oIpp \widetilde T^1}{e_i}^p - \h{\oEpp \widetilde T^1}{e_i}^p &= 1 - \frac 1 2 = \frac 1 2 , \\
	\h{\oIpp \widetilde T^1}{-e_i}^p - \h{\oEpp \widetilde T^1}{-e_i}^p &= 0 - 0 = 0.
\end{split}
\end{equation}
Analogously we calculate:
\begin{equation}
\begin{split}\label{eq: constant oIpp - oEpp + oJpp - oFpp T^2 +-e_i}
	\h{\oIpp T^2}{e_i}^p - \h{\oEpp T^2}{e_i}^p + \oJpp(T^2)(e_i) - \oFpp(T^2)(e_i) &= \frac 1 2 , \\
	\h{\oIpp T^2}{-e_i}^p - \h{\oEpp T^2}{-e_i}^p + \oJpp(T^2)(-e_i) - \oFpp(T^2)(-e_i) &= 0 .
\end{split}
\end{equation}
Similar to the proof of \th\ref{cP^2_o not continuous} we see that every linear combination of the four operators
\begin{equation*}
\begin{split}
	P &\mapsto \h{\oEpp P}{.}^p , \\
	P &\mapsto \h{\oEpm P}{.}^p , \\
	P &\mapsto \h{\oEpp P}{.}^p + \oFpp P
\end{split}
\end{equation*}
and
\begin{equation*}
	P \mapsto \h{\oEpm P}{.}^p + \oFpm P
\end{equation*}
which is continuous at $[o,e_1]$ is actually a linear combination of $\oFpp$ and $\oFpm$.
Clearly
\begin{equation*}
	\lim_{\varepsilon \to 0} \varepsilon e_2 + [o,e_1] = [o, e_1] .
\end{equation*}
Since
\begin{equation*}
	  \lim_{\varepsilon \to 0} \oFpp(\varepsilon e_2 + [o,e_1])(e_1)
	= \lim_{\varepsilon \to 0} \frac 1 2
	= \frac 1 2
\end{equation*}
and since
\begin{equation*}
	\oFpp([o, e_1])(e_1) = 0 ,
\end{equation*}
we see that $\oFpp$ is not continuous at $[o,e_1]$.
Similarly we see that the only linear combination of the operators $\oFpp$ and $\oFpm$
which is continuous at $[o,e_1]$ is the trivial one.
Therefore we arrive at the following lemma.

\begin{lemma}\th\label{cP^2 not continuous}
	The only linear combination of
	\begin{equation*}
	\begin{split}
		P &\mapsto \h{\oEpp P}{.}^p , \\
		P &\mapsto \h{\oEpm P}{.}^p , \\
		P &\mapsto \h{\oEpp P}{.}^p + \oFpp P
	\end{split}
	\end{equation*}
	and
	\begin{equation*}
		P \mapsto \h{\oEpm P}{.}^p + \oFpm P
	\end{equation*}
	which is continuous at $[o,e_1]$ is the trivial one.
\end{lemma}


We complete this section with a simple lemma (cf. \cite{Parapatits_1}).

\begin{lemma}\th\label{sln covariance gln}
	Let $\cQ^n$ be either $\cP^n_o$ or $\cP^n$ and let $\oPhi \colon \cQ^n \to \Cp$ be $\sln$-covariant.
	Furthermore let $\phi \in \gln$ with $\det \phi > 0$.
	Then
	\begin{equation*}
		\oPhi(\phi P) = \det(\phi)^{-\frac p n} \oPhi \left( \det(\phi)^{\frac 1 n} P \right) \circ \phi^t
	\end{equation*}
	for all $P \in \cQ^n$.
\end{lemma}
\begin{proof}
	Since $\det(\phi)^{-\frac 1 n} \phi \in \sln$,
	this follows directly from the $\sln$-covariance of $\oPhi$ and the $p$-homogeneity of the functions in $\Cp$.
\end{proof}

%% file: proofPno.tex
The goal of this section is the classification
of all continuous $\sln$-covariant $\Lp$-Minkowski valuations on $\cK^n_o$.
It will be convenient to first prove a slightly more general theorem about valuations from $\cP^n_o$ to $\Cp$.
We begin by proving a classification with the additional assumption of simplicity.
The next theorem is an adaptation of a corresponding theorem concerning $\sln$-contravariant valuations from \cite{Parapatits_1}.

\begin{theorem}\th\label{class sln co simple}
	Let $n \geq 3$ and let $\oPhi \colon \cP^n_o \rightarrow \Cp$ be a simple $\sln$-covariant valuation.
	Assume further that for every $y \in \R^n$ there exists a bounded open interval $I_y \subseteq (0,+\infty)$
	such that $\{\oPhi(sT^n)(y): s \in I_y\}$ is not dense in $\R$.
	Then there exist constants $c_1, c_2 \in \R$ such that
	\begin{equation}
		\oPhi P = c_1 \h{\oMpp P}{.}^p + c_2 \h{\oMpm P}{.}^p
	\end{equation}
	for all $P \in \cP^n_o$.
\end{theorem}
\begin{proof}
	Using the $\sln$-covariance of $\oPhi$ and \th\ref{valuation determined simplices} it is enough to prove
	\begin{equation}\label{class sln co simple - simplices}
		\oPhi( s T^n ) = c_1 \h{\oMpp( s T^n )}{.}^p + c_2 \h{\oMpm( s T^n )}{.}^p
	\end{equation}
	for $s > 0$.

	\noindent\textbf{1. Functional Equation: }
	Let $\lambda \in (0,1)$ and denote by $H_\lambda$ the hyperplane through $o$ with normal vector
	$\lambda e_1-(1-\lambda)e_2$.
	Since $\oPhi$ is a valuation we get
	\begin{equation*}
		\oPhi(sT^n) + \oPhi(sT^n \cap H_\lambda) = \oPhi(sT^n \cap H_\lambda^+) + \oPhi(sT^n \cap H_\lambda^-) ,
	\end{equation*}
	where $H_\lambda^+$ and $H_\lambda^-$ are the two halfspaces bounded by $H_\lambda$.
	Because $\oPhi$ is assumed to be simple, we get
	\begin{equation}\label{class sln co simple - val property two}
		\oPhi(sT^n) = \oPhi(sT^n \cap H_\lambda^+) + \oPhi(sT^n \cap H_\lambda^-) .
	\end{equation}
	Define $\phi_\lambda \in \gln$ by
		$$\phi_\lambda e_1 = e_1 ,\quad
		  \phi_\lambda e_2 = (1-\lambda)e_1 + \lambda e_2 ,\quad
		  \phi_\lambda e_k = e_k \quad \text{for } 3 \leq k \leq n$$
	and $\psi_\lambda \in \gln$ by
		$$\psi_\lambda e_1 = (1-\lambda)e_1 + \lambda e_2 ,\quad
		  \psi_\lambda e_2 = e_2 ,\quad
		  \psi_\lambda e_k = e_k \quad \text{for } 3 \leq k \leq n .$$
	Note that
	\begin{equation}\label{class sln co simple - det}
		\det(\phi_\lambda) = \lambda \quad \text{and} \quad \det(\psi_\lambda) = 1-\lambda .
	\end{equation}
	Since
		$$T^n \cap H_\lambda^+ = \phi_\lambda T^n \quad \text{and} \quad
		  T^n \cap H_\lambda^- = \psi_\lambda T^n ,$$
	Equation \eqref{class sln co simple - val property two} becomes
	\begin{equation*}
		\oPhi(sT^n) = \oPhi(s \phi_\lambda T^n) + \oPhi(s \psi_\lambda T^n) .
	\end{equation*}
	Using \th\ref{sln covariance gln} and \eqref{class sln co simple - det} we can rewrite the last equation as
	\begin{equation}\label{class sln co simple - functional equation}
		\oPhi(sT^n)(x) =
		\lambda^{-\frac p n} \oPhi\left( \lambda^{\frac 1 n} sT^n \right)(\phi_\lambda^t x)
		+ (1-\lambda)^{-\frac p n} \oPhi\left( (1-\lambda)^{\frac 1 n} sT^n \right)(\psi_\lambda^t x)
	\end{equation}
	for all $x \in \R^n$.

	\noindent\textbf{2. Homogeneity: }
	For $y \in \{e_1,e_2\}^\perp$ \eqref{class sln co simple - functional equation} becomes
	\begin{equation*}
		\oPhi(sT^n)(y) =
		\lambda^{-\frac p n} \oPhi\left( \lambda^{\frac 1 n} sT^n \right)(y) + (1-\lambda)^{-\frac p n} \oPhi\left( (1-\lambda)^{\frac 1 n} sT^n \right)(y) .
	\end{equation*}
	Replace $s$ with $s^{\frac 1 n}$ in the above equation and define $g(s) = \oPhi\left( s^{\frac 1 n} T^n \right)(y)$.
	Then we have
	\begin{equation*}
		g(s) = \lambda^{-\frac p n} g(\lambda s) + (1-\lambda)^{-\frac p n} g((1-\lambda) s) .
	\end{equation*}
	Let $a,b > 0$.
	We set $s = a+b$ and $\lambda = \frac {a} {a+b}$ to get
		$$ g(a+b) = \left( \frac {a} {a+b} \right)^{-\frac p n} g(a) + \left( \frac {b} {a+b} \right)^{-\frac p n} g(b) $$
	and hence
		$$ (a+b)^{-\frac p n}g(a+b) = a^{-\frac p n} g(a) + b^{-\frac p n} g(b) .$$
	We see that $s \mapsto s^{-\frac p n} g(s)$ solves Cauchy's functional equation for $s > 0$.
	By assumption there is a bounded open interval $I_y$ such that $g(I_y)$ is not dense in $\R$.
	It follows that $s \mapsto s^{-\frac p n} g(s)$ is linear.
	This implies $s^{-\frac p n} g(s) = s g(1)$ and hence $g(s) = s^{1+\frac p n} g(1)$.
	The definition of $g$ yields
	\begin{equation*}
		\oPhi(s T^n)(y) = g(s^n) = s^{n+p} g(1) = s^{n+p} \oPhi(T^n)(y) .
	\end{equation*}
	Since $n \geq 3$ and since we can do the above calculation for any two standard basis vectors, we obtain in particular
	\begin{equation}\label{class sln co simple - homogeneity}
		\oPhi(s T^n)(\pm e_i) = s^{n+p} \oPhi(T^n)(\pm e_i) \quad \text{for } i = 1,\ldots,n .
	\end{equation}

	\noindent\textbf{3. Constants: }
	Let $i \in \{1,\ldots,n\}$.
	Since $n \geq 3$, we can find a permutation of the coordinates $\phi \in \sln$ such that $\phi^t e_1 = e_i$.
	It follows that
	\begin{equation}\label{class sln co simple - values on standard basis}
		\oPhi(T^n)(e_i) = \oPhi(T^n)(\phi^t e_1) = \oPhi( \phi T^n)(e_1) = \oPhi(T^n)(e_1) .
	\end{equation}
	Similarly we get $\oPhi(T^n)(-e_i) = \oPhi(T^n)(-e_1)$.
	Set
	\begin{equation}\label{class sln co simple - constants}
		c_1 = \frac {\Gamma(p+1+n)}{\Gamma(p+1)} \oPhi(T^n)(e_1) \quad \text{and} \quad
		c_2 = \frac {\Gamma(p+1+n)}{\Gamma(p+1)} \oPhi(T^n)(-e_1) .
	\end{equation}
	
	\noindent\textbf{4. Induction: }
	We are now going to show by induction on the number $m$ of coordinates of $x$ not equal to zero that
	\begin{equation}\label{class sln co simple - induction}
		\oPhi( s T^n )(x) = c_1 \h{\oMpp( sT^n )}{x}^p + c_2 \h{\oMpm( s T^n)}{x}^p
	\end{equation}
	for $s > 0$ and for all $x \in \R^n$.
	Note that since $P \mapsto c_1 \h{\oMpp(P)}{.}^p + c_2 \h{\oMpm(P)}{.}^p$ satisfies the assumptions of the theorem it also satisfies
	\eqref{class sln co simple - functional equation} and \eqref{class sln co simple - homogeneity}.
	
	The case $m = 0$ is trivial.
	The case $m = 1$ is also easy to verify with \eqref{class sln co simple - homogeneity},
	\eqref{class sln co simple - values on standard basis}, \eqref{class sln co simple - constants}, \eqref{eq: constant oMpp T^n e_i} and \eqref{eq: constant oMpp T^n -e_i}.
	Now, let $m \geq 2$.
	Without loss of generality assume that $x_1, x_2 \neq 0$ and $|x_1| \geq |x_2|$.
	Since functions in $\Cp$ are continuous, we can further assume that $|x_1| > |x_2|$.

	First consider the case that $x_1$ and $x_2$ have different signs.
	Set $\lambda = \frac {x_1} {x_1 - x_2} \in (0,1)$ and calculate
	\begin{align*}
		   \phi_\lambda^t x
		&= x_1 e_1 + x_1 (1-\lambda) e_2 + x_2 \lambda e_2 + x_3 e_3 + \ldots + x_n e_n \\
		&= x_1 e_1 + x_3 e_3 + \ldots + x_n e_n .
	\end{align*}
	Similarly we have
	\begin{equation*}
		\psi_\lambda^t x = x_2 e_2 + x_3 e_3 + \ldots + x_n e_n .
	\end{equation*}
	Using \eqref{class sln co simple - functional equation} and the induction hypotheses gives the desired result.
	
	Now consider the case that $x_1,x_2$ have the same sign.
	Set $\lambda = 1- \frac {x_2} {x_1} \in (0,1)$ and calculate
	\begin{align*}
		&\phantom{={}}\ \phi_\lambda^t (x_1 e_1 + x_3 e_3 + \ldots + x_n e_n) \\
		&= x_1 e_1 + x_1 (1-\lambda) e_2 + x_3 e_3 + \ldots + x_n e_n \\
		&= x_1 e_1 + x_2 e_2 + x_3 e_3 + \ldots + x_n e_n \\
		&= x
	\end{align*}
	or equivalently
	\begin{equation*}
		\phi_\lambda^{-t} x = x_1 e_1 + x_3 e_3 + \ldots + x_n e_n .
	\end{equation*}
	Similarly we calculate
	\begin{equation*}
		\psi_\lambda^t \phi_\lambda^{-t} x = x_1 (1-\lambda) e_1 + x_3 e_3 + \ldots + x_n e_n .
	\end{equation*}
	Using \eqref{class sln co simple - functional equation} with $x$ replaced by $\phi_\lambda^{-t} x$ and using the induction hypotheses gives the desired result.

	This completes the induction and proves \eqref{class sln co simple - induction} or equivalently \eqref{class sln co simple - simplices}.
\end{proof}

\begin{remark}
	In the second step of the preceeding proof we made critical use of the assumption that
	the dimension $n$ is greater or equal than $3$.
	This is the main obstacle for establishing a similar result for $n = 2$.
	Ludwig's results \cite{Ludwig2005}, however, also hold for $n = 2$.
\end{remark}

We now use \th\ref{class sln co simple} to rule out the existence of certain valuations.
This will be needed in the induction step in the proof of \th\ref{class sln co}.

\begin{lemma}\th\label{sln co simple p homogeneous}
	Let $n \geq 2$.
	If $\oPhi \colon \cP^n_o \rightarrow \Cp$ is a simple $\gln$-covariant valuation
	which is continuous at the line segment $[o,e_1]$, then $\oPhi = 0$.
\end{lemma}
\begin{proof}
	First note that the $\gln$-covariance implies $p$-homogeneity.
	For $n \geq 3$ the assertion follows directly from \th\ref{class sln co simple},
	since $P \mapsto \h{\oMpp P}{.}^p$ and $P \mapsto \h{\oMpm P}{.}^p$ are both $(n+p)$-homogeneous.
	
	Only the case $n = 2$ remains.
	Using the $\gln$-covariance and \th\ref{valuation determined simplices}
	we see that $\oPhi$ is determined by its values on $T^2$.	
	The $\gln$-covariance and the simplicity of $\oPhi$ imply
	\begin{equation*}
		\oPhi(T^2)(x) = \oPhi( T^2 )(\phi_\lambda^t x) + \oPhi( T^2 )(\psi_\lambda^t x) .
	\end{equation*}
	Similar to the steps 3 and 4 in the proof of \th\ref{class sln co simple}, we see that
	$\oPhi$ is determined by the two values $\oPhi(T^2)(\pm e_1)$.
	
	Now notice that $\oEpp$ takes the same values as $\oIpp$
	for all $P \in \cP^2_o$ with $\dim P \leq 1$.
	Therefore
	\begin{equation*}
		P \mapsto \h{\oIpp P}{.}^p - \h{\oEpp P}{.}^p
	\end{equation*}
	is a simple $\gln$-covariant valuation.
	The same holds for
	\begin{equation*}
		P \mapsto \h{\oIpm P}{.}^p - \h{\oEpm P}{.}^p .
	\end{equation*}
	Using \eqref{eq: constants oIpp - oEpp T^2 +-e_i} we see that $\oPhi$ is a linear combination of these two operators.
	By \th\ref{cP^2_o not continuous} and the fact that $\oIpp$ and $\oIpm$ are continuous
	the only operator in the linear hull of these operators which is continuous at the line segment $[o,e_1]$ is $\oPhi = 0$.
\end{proof}

The next lemma will also be needed in the proof of \th\ref{class sln co}.
It implies that an $\sln$-covariant operator $\oPhi \colon \cP^n_o \to \cK^n_o$
maps a convex polytope which is contained in some linear subspace
to a convex body which is contained in the same linear subspace.

\begin{lemma}\th\label{oPhi P subseteq lin P}
	Let $n \geq 2$.
	If $\oPhi \colon \cP^n_o \to \Cp$ is $\sln$-covariant, then
	\begin{equation*}
		\oPhi(P)(x) = \oPhi(P)(\pi_P x)
	\end{equation*}
	for all $P \in \cP^n_o$ and $x \in \R^n$,
	where $\pi_P$ denotes the orthogonal projection onto $\lin P$.
\end{lemma}
\begin{proof}
	For $\dim P = n$ there is nothing to show.
	For $\dim P = 0$, i.e.\ $P = \{o\}$, we have
	\begin{equation*}
		\oPhi P = \oPhi (\phi P) = \oPhi (P) \circ \phi^t
	\end{equation*}
	for all $\phi \in \sln$.
	Since $n \geq 2$, $\oPhi P$ must be constant.

	Now let $d = \dim P$ with $0<d<n$.
	Using the $\sln$-covariance we can assume without loss of generality that $P \subseteq \{e_1,\ldots,e_d\}$.
	Define $\phi \in \sln$ by
	\begin{equation*}
		\phi =
		\begin{pmatrix}
			I' & A \\
			0 & I'' \\
		\end{pmatrix} ,
	\end{equation*}
	where $A \in \R^{d \times (n-d)}$ is an arbitrary matrix,
	$0 \in \R^{(n-d) \times d}$ is the zero matrix
	and where $I' \in \R^{d \times d}$ and $I'' \in \R^{(n-d) \times (n-d)}$ are identity matrices.
	Since $\phi P = P$ and since $\oPhi$ is $\sln$-covariant, we get
	\begin{equation}\label{oPhi P subseteq lin P - eq one}
		\oPhi P = \oPhi (\phi P) = \oPhi (P) \circ \phi^t .
	\end{equation}
	Write $x = (x', x'') \in \R^d \times \R^{n-d}$.
	Since $\oPhi P$ is continuous, we can assume that $x'$ is not zero.
	Note that
	\begin{equation}\label{oPhi P subseteq lin P - eq two}
		\phi^t x =
		\begin{pmatrix}
			I' & 0^t \\
			A^t & I'' \\
		\end{pmatrix} \cdot
		\begin{pmatrix}
				x' \\
				x'' \\
		\end{pmatrix} =
		\begin{pmatrix}
				x' \\
				A^t x' + x'' \\
		\end{pmatrix} .
	\end{equation}
	Because we can choose $A$ such that $A^t x' + x''$ is zero,
	the assertion follows from \eqref{oPhi P subseteq lin P - eq one} and \eqref{oPhi P subseteq lin P - eq two}.
\end{proof}

Now we are able to classify all $\sln$-covariant valuations from $\cP^n_o$ to $\Cp$ which satisfy
certain continuity properties.

\begin{theorem}\th\label{class sln co}
	Let $n \geq 3$ and let $\oPhi \colon \cP^n_o \to \Cp$ be an $\sln$-covariant valuation.
	Assume further that for every $y \in \R^n$ there exists a bounded open interval $I_y \subseteq (0,+\infty)$
	such that $\{\oPhi(sT^n)(y): s \in I_y\}$ is not dense in $\R$.
	Also assume that $\oPhi$ is continuous at the line segment $[o,e_1]$.
	Then there exist constants $c_1,c_2,c_3,c_4 \in \R$ such that
	\begin{equation}
		\oPhi P = c_1 \h{\oMpp P}{.}^p + c_2 \h{\oMpm P}{.}^p + c_3 \h{\oIpp P}{.}^p + c_4 \h{\oIpm P}{.}^p
	\end{equation}
	for all $P \in \cP^n_o$.
\end{theorem}
\begin{proof}
	\th\ref{oPhi P subseteq lin P} and the $p$-homogeneity of functions in $\Cp$ show that
	\begin{equation}\label{class sln co - constants}
		\oPhi([o,e_1])(x) = \oPhi([o,e_1])(x_1 e_1) = |x_1|^p \oPhi([o,e_1])(\sgn (x_1) e_1)
	\end{equation}
	for all $x \in \R^n$.
	Set $c_3 = \oPhi([o,e_1])(e_1)$ and $c_4 = \oPhi([o,e_1])(-e_1)$.
	Define $\oPsi \colon \cP^n_o \rightarrow \Cp$ by
	\begin{equation}
		\oPsi P = \oPhi P - c_3 \h{\oIpp P}{.}^p - c_4 \h{\oIpm P}{.}^p
	\end{equation}
	for all $P \in \cP^n_o$.
	Note that $\oPsi$ is also an $\sln$-covariant valuation.
	If we can show that $\oPsi$ is simple, then the assertion follows from \th\ref{class sln co simple}.
	
	We need to prove that $\oPsi P = 0$ for all $P \in \cP^n_o$ with $\dim P \leq n-1$.
	Using the $\sln$-covariance we can assume without loss of generality that $P \subseteq \lin\{e_1,\ldots,e_{n-1}\}$.
	We will use induction on $d = 1,\ldots,n-1$ to show that
	$\oPsi P = 0$ for all convex polytopes $P \subseteq \lin\{e_1,\ldots,e_d\}$.
	Define $\hat\oPsi \colon \cP^d_o \rightarrow \Cp[d]$ by
	\begin{equation}
		\hat\oPsi P = \oPsi(\iota_d P) \circ \iota_d
	\end{equation}
	for all $P \in \cP^d_o$, where $\iota_d$ denotes the natural embedding of $\R^d$ in $\R^n$.
	It is easy to see that $\hat\oPsi$ is a $\gln[d]$-covariant valuation.
	Since $\oPsi$ is $\sln$-covariant, we have $\oPsi(\{o\}) = 0$ by \th\ref{oPhi P subseteq lin P}.
	For $d = 1$ the induction statement follows from \eqref{class sln co - constants}, \eqref{eq: constant oIpp [o,e_i]}
	and the $\gln[d]$-covariance.
	For $2 \leq d \leq n-1$ the induction statement follows from
	the induction hypothesis, \th\ref{sln co simple p homogeneous} and \th\ref{oPhi P subseteq lin P}.
	This finishes the induction and completes the proof of the theorem.
\end{proof}

Finally we can prove our desired result about continuous $\sln$-covariant $\Lp$-Minkowski valuations on $\cK^n_o$.
This is one direction of the first main theorem from the introduction.
The other direction is trivial.

\begin{corollary}\th\label{class sln co mink}
	Let $n \geq 3$.
	If $\oPhi \colon \cK^n_o \to \cK^n_o$ is a continuous $\sln$-covariant $\Lp$-Minkowski valuation,
	then there exist constants $c_1, c_2, c_3, c_4 \geq 0$ such that
	\begin{equation*}
		\oPhi K = c_1 \oMpp K \padd c_2 \oMpm K \padd c_3 \oIpp K \padd c_4 \oIpm K
	\end{equation*}
	for all $K \in \cK^n_o$.
\end{corollary}
\begin{proof}
	The map
	\begin{equation*}
		P \mapsto \h{\oPhi P}{.}^p, \quad P \in \cP^n_o
	\end{equation*}
	satisfies the assumptions of \th\ref{class sln co}.
	Thus, we get constants $d_1, d_2, d_3, d_4 \in \R$ such that
	\begin{equation*}
		\h{\oPhi P}{.}^p = d_1 \h{\oMpp P}{.}^p + d_2 \h{\oMpm P}{.}^p + d_3 \h{\oIpp P}{.}^p + d_4 \h{\oIpm P}{.}^p
	\end{equation*}
	holds for all $P \in \cP^n_o$.
	By continuity this is true for all $K \in \cK^n_o$.
	It remains to show that $d_1, d_2, d_3, d_4 \geq 0$.
	To this end consider for fixed $K \in \cK^n_o$ and $x \in \R^n$ the map $s \mapsto \h{\oPhi(sK)}{x}^p, \ s > 0$.
	Using the homogeneity of the operators $\oMpp$, $\oMpm$, $\oIpp$ and $\oIpm$ we get
	\begin{align*}
		0
		&\leq \h{\oPhi(sK)}{x}^p \\
		&= \left( d_1 \h{\oMpp K}{x}^p + d_2 \h{\oMpm K}{x}^p \right) s^{n+p}
		 + \left( d_3 \h{\oIpp K}{x}^p + d_4 \h{\oIpm K}{x}^p \right) s^p .
	\end{align*}
	Dividing this equation by $s^p$ and letting $s \to 0$ we see that
	\begin{equation*}
		0 \leq d_3 \h{\oIpp K}{x}^p + d_4 \h{\oIpm K}{x}^p.
	\end{equation*}
	Setting $K = [o,e_1]$ and $x = \pm e_1$ and using \eqref{eq: constant oIpp [o,e_i]} shows that $d_3,d_4 \geq 0$.
	Similarly we see that $d_1, d_2 \geq 0$.
	Defining $c_i = \sqrt[p]{d_i}$ for $i = 1,\ldots,4$ finishes the proof.
\end{proof}

%% file: proofPn.tex
In this section our goal is the classification
of all continuous $\sln$-covariant $\Lp$-Minkowski valuations on $\cK^n$.
We start with a classification of valuations from $\cP^n$ to $\Cp$
with the additional assumption that the valuation is almost simple.
Let $\cQ^n$ be a subset of $\cK^n$ and $A$ an abelian monoid,
we call a valuation $\oPhi \colon \cQ^n \to A$ almost simple, if $\oPhi K = 0$ for all $K \in \cQ^n$
with $\dim K \leq n-2$ and for all $K \in \cQ^n$ with $\dim K = n-1$ and $o \in \aff K$.

\begin{theorem}\th\label{class sln co almost simple Pn}
	Let $n \geq 3$ and let $\oPhi \colon \cP^n \rightarrow \Cp$ be an almost simple $\sln$-covariant valuation.
	Assume further that for every $y \in \R^n$ there exists a bounded open interval $I_y \subseteq (0,+\infty)$
	such that $\{\oPhi(sT^n)(y): s \in I_y\}$ is not dense in $\R$
	and that for every $y \in \R^n$ there exists a bounded open interval $J_y \subseteq (0,+\infty)$
	such that $\{\oPhi(s \widetilde T^{n-1})(y): s \in J_y\}$ is not dense in $\R$.
	Then there exist constants $c_1, c_2, c_3, c_4 \in \R$ such that
	\begin{equation*}
		\oPhi P = c_1 \h{\oMpp P}{.}^p + c_2 \h{\oMpm P}{.}^p + c_3 \h{\oMpps P}{.}^p + c_4 \h{\oMpms P}{.}^p
	\end{equation*}
	for all $P \in \cP^n$.
\end{theorem}
\begin{proof}
	By replacing $T^n$ with $\widetilde T^{n-1}$ and by replacing
	\eqref{eq: constant oMpp T^n e_i} and \eqref{eq: constant oMpp T^n -e_i}
	with \eqref{eq: constant oMpps widetilde T^(n-1) e_i} and \eqref{eq: constant oMpps widetilde T^(n-1) -e_i}, respectively,
	in the proof of \th\ref{class sln co simple} we see that there exist constants $c_3, c_4$ such that
	\begin{equation*}
		\oPhi( s \widetilde T^{n-1} )
		= c_3 \h{\oMpps( s \widetilde T^{n-1} )}{.}^p + c_4 \h{\oMpms( s \widetilde T^{n-1} )}{.}^p
	\end{equation*}
	for $s > 0$.
	Note that the constants are given by
	\begin{equation*}
		c_3 = \frac {\Gamma(p+1+n)}{\Gamma(p+1)} \oPhi(\widetilde T^{n-1})(e_1) \quad \text{and} \quad
		c_4 = \frac {\Gamma(p+1+n)}{\Gamma(p+1)} \oPhi(\widetilde T^{n-1})(-e_1) .
	\end{equation*}
	Define $\oPsi \colon \cP^n \rightarrow \Cp$ by
	\begin{equation*}
		\oPsi P = \oPhi P - c_3 \h{\oMpps P}{.}^p - c_4 \h{\oMpms P}{.}^p
	\end{equation*}
	for all $P \in \cP^n$.
	Note that $\oPsi$ is an $\sln$-covariant valuation.
	We use $\oPsi( s \widetilde T^{n-1}) = 0$, the $\sln$-covariance of $\oPsi$,
	\th\ref{incl excl Pn} and the assumption that $\oPhi$ is almost simple to see that $\oPsi$ is simple.
	Now \th\ref{class sln co simple} implies that there exist constants $c_1, c_2 \in \R$ such that
	\begin{equation}\label{class sln co almost simple Pn - oPsi on Pno}
		\oPsi P = c_1 \h{\oMpp P}{.}^p + c_2 \h{\oMpm P}{.}^p
	\end{equation}
	for all $P \in \cP^n_o$.
	Because $\oMpp$ and $\oMpm$ are simple valuations on $\cP^n$,
	\th\ref{simple determined Pno} implies that \eqref{class sln co almost simple Pn - oPsi on Pno}
	holds for all $P \in \cP^n$.
	Using the definition of $\oPsi$ finishes the proof.
\end{proof}

The next lemma is the analog of \th\ref{sln co simple p homogeneous} for $\cP^n$
and will be needed in the induction step in the proof of \th\ref{class sln co Pn}.
Again, it will rule out the existence of certain valuations.

\begin{lemma}\th\label{sln co almost simple p homogeneous Pn}
	Let $n \geq 2$.
	If $\oPhi \colon \cP^n \rightarrow \Cp$ is an almost simple $\gln$-covariant valuation
	which is continuous at the line segment $[o,e_1]$, then $\oPhi = 0$.
\end{lemma}
\begin{proof}
	First note that the $\gln$-covariance implies $p$-homogeneity.
	For $n \geq 3$ the assertion follows directly from \th\ref{class sln co almost simple Pn},
	since $P \mapsto \h{\oMpp P}{.}^p$, $P \mapsto \h{\oMpm P}{.}^p$, $P \mapsto \h{\oMpps P}{.}^p$
	and $P \mapsto \h{\oMpms P}{.}^p$ are all $(n+p)$-homogeneous.
	
	Only the case $n = 2$ remains to be proved.
	Similar to the steps 3 and 4 in the proof of \th\ref{class sln co simple}, we see that
	$\oPhi \widetilde T^1$ is determined by the two values $\oPhi(\widetilde T^1)(\pm e_1)$.
	Now notice that
	\begin{equation*}
		P \mapsto \h{\oIpp P}{.}^p - \h{\oEpp P}{.}^p \quad \text{and} \quad
		P \mapsto \h{\oIpm P}{.}^p - \h{\oEpm P}{.}^p,
	\end{equation*}
	are almost simple $\gln$-covariant valuations.
	Using \eqref{eq: constant oIpp - oEpp widetilde T^1 +-e_i}, the $\gln$-covariance and \th\ref{incl excl Pn} we see that
	by subtracting a suitable linear combination of these two operators from $\oPhi$
	we get a simple $\gln$-covariant valuation $\oPsi$.
	Using the same arguments as in the proof of \th\ref{class sln co simple} and using \th\ref{simple determined Pno}
	we see that $\oPsi$ is determined by the two values $\oPhi(T^2)(\pm e_1)$.
	Since
		$$ P \mapsto \h{\oIpp P}{.}^p - \h{\oEpp P}{.}^p + \oJpp P - \oFpp P $$
	and
		$$ P \mapsto \h{\oIpm P}{.}^p - \h{\oEpm P}{.}^p + \oJpm P - \oFpm P $$
	are simple $\gln$-covariant valuations, we can use \eqref{eq: constant oIpp - oEpp + oJpp - oFpp T^2 +-e_i} to conclude that
	$\oPsi$ is a linear combination of these two operators.
	Therefore, $\oPhi$ is a linear combination of the four operators above.
	By \th\ref{cP^2 not continuous} and the fact that $\oIpp$, $\oIpm$, $\oJpp$ and $\oJpm$ are continuous
	the only operator in the linear hull of these operators which is continuous at the line segment $[o,e_1]$ is $\oPhi = 0$.
\end{proof}

\begin{lemma}\th\label{oPhi P subseteq lin P Pn}
	Let $n \geq 2$.
	If $\oPhi \colon \cP^n \rightarrow \Cp$ is $\sln$-covariant,
	then $\oPhi(P)(x) = \oPhi(P)(\pi_P x)$ for all $P \in \cP^n$ and $x \in \R^n$,
	where $\pi_P$ denotes the orthogonal projection onto $\lin P$.
\end{lemma}
\begin{proof}
	The proof is similar to the proof of \th\ref{oPhi P subseteq lin P}.
\end{proof}

Now we are able to classify all $\sln$-covariant valuations from $\cP^n$ to $\Cp$ which satisfy
certain continuity properties.

\begin{theorem}\th\label{class sln co Pn}
	Let $n \geq 3$ and let $\oPhi \colon \cP^n \rightarrow \Cp$ be an $\sln$-covariant valuation.
	Assume further that for every $y \in \R^n$ there exists a bounded open interval $I_y \subseteq (0,+\infty)$
	such that $\{\oPhi(sT^n)(y): s \in I_y\}$ is not dense in $\R$
	and that for every $y \in \R^n$ there exists a bounded open interval $J_y \subseteq (0,+\infty)$
	such that $\{\oPhi(s \widetilde T^{n-1})(y): s \in J_y\}$ is not dense in $\R$.
	Also assume that $\oPhi$ is continuous at the line segment $[o,e_1]$.
	Then there exist constants $c_i \in \R \ ,i=1\ldots,8$ such that
	\begin{align*}
		\oPhi P &= c_1 \h{\oMpp P}{.}^p + c_2 \h{\oMpm P}{.}^p + c_3 \h{\oMpps P}{.}^p + c_4 \h{\oMpms P}{.}^p \\
		        &+ c_5 \h{\oIpp P}{.}^p + c_6 \h{\oIpm P}{.}^p + c_7 \oJpp P + c_8 \oJpm P
	\end{align*}
	for all $P \in \cP^n$.
\end{theorem}
\begin{proof}
	\th\ref{oPhi P subseteq lin P Pn} and the $p$-homogeneity of the functions in $\Cp$ show that
	\begin{equation}\label{class sln co Pn - constants 1}
		\oPhi([o,e_1])(x) = \oPhi([o,e_1])(x_1 e_1) = |x_1|^p \oPhi([o,e_1])(\sgn (x_1) e_1)
	\end{equation}
	and
	\begin{equation}\label{class sln co Pn - constants 2}
		\oPhi([e_1, 2 e_1])(x) = |x_1|^p \oPhi([e_1, 2 e_1])(\sgn (x_1) e_1)
	\end{equation}
	for all $x \in \R^n$.
	Set $c_5 = \oPhi([o,e_1])(e_1)$, $c_6 = \oPhi([o,e_1])(-e_1)$,
	$c_7 = \oPhi([e_1, 2 e_1])(e_1) - c_5 2^p$ and $c_8 = \oPhi([e_1, 2 e_1])(-e_1) - c_6 2^p$.
	Define $\oPsi \colon \cP^n \rightarrow \Cp$ by
	\begin{equation}
		\oPsi P = \oPhi P - c_5 \h{\oIpp P}{.}^p - c_6 \h{\oIpm P}{.}^p - c_7 \oJpp P - c_8 \oJpm P
	\end{equation}
	for all $P \in \cP^n$.
	Note that $\oPsi$ is also an $\sln$-covariant valuation.
	If we can show that $\oPsi$ is almost simple, then the assertion follows from \th\ref{class sln co almost simple Pn}.
	
	We need to prove that $\oPsi P = 0$ for all $P \in \cP^n$ with $\dim P \leq n-2$
	and for all $P \in \cP^n$ with $\dim P = n-1$ and $o \in \aff P$.
	Using the $\sln$-covariance we can assume without loss of generality that $P \subseteq \lin\{e_1,\ldots,e_{n-1}\}$.
	We will use induction on $d = 1,\ldots,n-1$ to show that
	$\oPsi P = 0$ for all convex polytopes $P \subseteq \lin\{e_1,\ldots,e_d\}$.
	Define $\hat\oPsi \colon \cP^d \rightarrow \Cp[d]$ by
	\begin{equation}
		\hat\oPsi P = \oPsi(\iota_d P) \circ \iota_d
	\end{equation}
	for all $P \in \cP^d$, where $\iota_d$ denotes the natural embedding of $\R^d$ in $\R^n$.
	It is easy to see that $\hat\oPsi$ is a $\gln[d]$-covariant valuation.
	Since $\oPsi$ is $\sln$-covariant, we have $\oPsi(\{o\}) = 0$ by \th\ref{oPhi P subseteq lin P Pn}.
	For $d = 1$ the induction statement follows from \eqref{class sln co Pn - constants 1}, \eqref{class sln co Pn - constants 2},
	\eqref{eq: constant oIpp [o,e_i]}, \eqref{eq: constant oIpp [e_i, 2e_i]}, \eqref{eq: constant oJpp([e_i, 2e_i])},
	the $\gln[d]$-covariance and \th\ref{oPhi P subseteq lin P Pn}.
	For $2 \leq d \leq n-1$ the induction statement follows from
	the induction hypothesis, \th\ref{sln co almost simple p homogeneous Pn} and \th\ref{oPhi P subseteq lin P Pn}.
	This finishes the induction and completes the proof of the theorem.
\end{proof}

Finally we can prove our desired result about continuous $\sln$-covariant $\Lp$-Minkowski valuations on $\cK^n$.
This is one direction of the second main theorem from the introduction.
The other direction is trivial.

\begin{corollary}\th\label{class sln co Pn mink}
	Let $n \geq 3$.
	If $\oPhi \colon \cK^n \rightarrow \cK^n_o$ is a continuous $\sln$-covariant $\Lp$-Minkowski valuation,
	then there exist constants $c_1,c_2,c_3,c_4,c_5,c_6 \geq 0$ such that
	\begin{equation*}
		\oPhi K = c_1 \oMpp K \padd c_2 \oMpm K \padd c_3 \oMpps K \padd c_4 \oMpms K \padd c_5 \oIpp K \padd c_6 \oIpm K
	\end{equation*}
	for all $K \in \cK^n$.
\end{corollary}
\begin{proof}
	The map
	\begin{equation*}
		P \mapsto \h{\oPhi P}{.}^p, \quad P \in \cP^n_o
	\end{equation*}
	satisfies the assumptions of \th\ref{class sln co Pn}.
	Thus, we get constants $d_1,\ldots,d_8 \in \R$ such that
	\begin{align*}
		\h{\oPhi P}{.}^p
		&= d_1 \h{\oMpp P}{.}^p + d_2 \h{\oMpm P}{.}^p + d_3 \h{\oMpps P}{.}^p + d_4 \h{\oMpms P}{.}^p \\
		&+ d_5 \h{\oIpp P}{.}^p + d_6 \h{\oIpm P}{.}^p + d_7 \oJpp P + d_8 \oJpm P
	\end{align*}
	holds for all $P \in \cP^n$.
	By continuity this holds for all $K \in \cK^n$.
	It remains to show that $d_1,\ldots,d_6 \geq 0$ and that $d_7,d_8 = 0$.
	Similar to the proof of \th\ref{class sln co mink} we consider the map $s \mapsto \h{\oPhi(sK)}{.}, \ s > 0$
	for fixed $K \in \cK^n$ to see that both
	\begin{equation*}
		\left( d_1 \h{\oMpp K}{.}^p + d_2 \h{\oMpm K}{.}^p + d_3 \h{\oMpps K}{.}^p + d_4 \h{\oMpms K}{.}^p \right)^{\frac 1 p}
	\end{equation*}
	and
	\begin{equation}\label{class sln co Pn mink - one}
		\left( d_5 \h{\oIpp K}{.}^p + d_6 \h{\oIpm K}{.}^p + d_7 \oJpp K + d_8 \oJpm K \right)^{\frac 1 p}
	\end{equation}
	are nonnegative and sublinear.
	Using \eqref{eq: constant oMpp T^n e_i}, \eqref{eq: constant oMpp T^n -e_i},
	\eqref{eq: constant oMpps widetilde T^(n-1) e_i}, \eqref{eq: constant oMpps widetilde T^(n-1) -e_i}
	and \eqref{eq: constant oIpp [o,e_i]} we get that $d_1,\ldots,d_6 \geq 0$.
	Now set $K = \conv \{ e_1, e_2, e_1+e_2\}$ in \eqref{class sln co Pn mink - one}.
	Inserting the directions $e_1$, $e_2$ and $e_1+e_2$
	and using the subadditivity and \eqref{eq: constant oIpp oJpp conv (e_1, e_2, e_1+e_2)} gives us
	\begin{equation*}
		\left( 2^p d_5 + d_7 \right)^{\frac 1 p}
		\leq \left( d_5 \right)^{\frac 1 p} + \left( d_5 \right)^{\frac 1 p}
		= 2 \left( d_5 \right)^{\frac 1 p} .
	\end{equation*}
	Therefore we get $d_7 \leq 0$.
	Meanwhile, setting $K = \widetilde T^1$ in \eqref{class sln co Pn mink - one} and
	inserting the directions $e_1$, $e_1+e_2$ and $2e_1+e_2$ gives us by \eqref{eq: constant oIpp oJpp widetilde T^1}
	\begin{equation*}
		0 \leq \left( d_5 \right)^{\frac 1 p} + \left( d_5 + d_7 \right)^{\frac 1 p} - \left( 2^p d_5 + d_7\right)^{\frac 1 p} .
	\end{equation*}
	The right hand side of this equation is equal to $0$ for $d_7 = 0$.
	By differentiating the right hand side with respect to $d_7$
	and by using $p > 1$ we see that it is strictly monotone.
	Therefore we get $d_7 \geq 0$.
	The last two arguments show that $d_7 = 0$.
	Similarly we see that $d_8 = 0$.
	Setting $c_i = \sqrt[p]{d_i}$ for $i = 1,\ldots,6$ finishes the proof.
\end{proof}

%% file: acknowledgements.tex
The work of the author was supported by the Austrian Science Fund (FWF) project P22388-N13
and the Austrian Science Fund (FWF) project P23639-N18.
During the revision stage of the article, the work of the author was supported by the European Research Council (ERC),
within the Starting Grant project ``Isoperimetric inequalities and integral geometry'', project number: 306445.